\documentclass[12pt,oneside,reqno]{amsart}
\usepackage{hyperref,  amsmath, amssymb, graphicx, ifthen}
\hypersetup{
colorlinks = true,
linkcolor = red,
citecolor = blue,
pdftitle = Polygonal VH Complexes,
pdfauthor = Jason K.C. Polak and Daniel T. Wise,
pdfsubject = Group Theory
}
\newcommand{\vh}{\mathcal{VH}}

\newcommand{\iv}{^{-1}}

\newcommand{\field}[1]{\mathbb{#1}}
\newcommand{\integers}{\ensuremath{\field{Z}}}

\newcommand{\Euclidean}{\ensuremath{\field{E}}}

\DeclareMathOperator{\link}{link}

\newcommand{\showcomments}{yes}
\renewcommand{\showcomments}{no}

\newsavebox{\commentbox}
{\ifthenelse{\equal{\showcomments}{yes}}%
{\footnotemark
        \begin{lrbox}{\commentbox}
        \begin{minipage}[t]{1.25in}\raggedright\sffamily\tiny
        \footnotemark[\arabic{footnote}]}
{\begin{lrbox}{\commentbox}}}%
{\ifthenelse{\equal{\showcomments}{yes}}%
{\end{minipage}\end{lrbox}\marginpar{\usebox{\commentbox}}}
{\end{lrbox}}}

\newtheorem{thm}{Theorem}[section]
\newtheorem{lem}[thm]{Lemma}

\theoremstyle{definition}

\newtheorem{defn}[thm]{Definition}

\newtheorem{exmp}[thm]{Example}

\begin{document}
\title{Polygonal $\vh$ complexes}
\author[J.K.C.Polak]
{Jason~K.C.~Pol\'ak }
           \address{Math. \& Stats.\\
                    McGill Univ. \\
                    Montreal, QC, Canada H3A 2K6 }
           \email{jason@jpolak.org}
\author[D.~T.~Wise$^\ddag$]
{Daniel T. Wise}
           \address{Math. \& Stats.\\
                    McGill Univ. \\
                    Montreal, QC, Canada H3A 2K6 }
           \email{wise@math.mcgill.ca}
\thanks{Supported by NSERC}

\begin{abstract}
Ian Leary inquires whether a class of hyperbolic finitely presented groups are residually finite. We answer in the affirmative by giving a systematic version of a construction in his paper, which shows that the standard $2$-complexes of these presentations have a $\vh$-structure. This structure induces a splitting of these groups, which together with hyperbolicity, implies that these groups are residually finite.
\end{abstract}
\date{\today}
\maketitle

\section{Introduction}
Recall that a group is \emph{residually finite} if every nontrivial element maps to a nontrivial element in a finite quotient. For instance, $GL_n(\integers)$ is residually finite, and hence so too are its subgroups. Ian Leary describes the following group $G$ in~\cite{Leary_KanThurston}:
\begin{equation}\label{eqn:leary1}
	\left\langle a,b,c,d,e,f  \bigg|
  \begin{array}{ccc}
     abcdef, & ab\iv c^2 f\iv e^2 d\iv, & a^2fc^2bed, \\
    ad^{-2}cb^{-2}ef\iv, & af^{-2}cd\iv eb^{-2}, & ad^2cf^2eb^2
  \end{array}
\right\rangle
\end{equation}
Leary shows that $G\cong \pi_1X$ where $X$ is a nonpositively curved square complex. He constructs $X$ by subdividing the standard 2-complex of Presentation~\eqref{eqn:leary1}. Leary asks whether $G$ and some similar groups are residually finite, and he reports that his investigations with the GAP software indicated that $G$ has few low-index subgroups. We note that $G$ is perfect -- which is part of the reason Leary was drawn to investigate groups like $G$ in conjunction with his Kan-Thurston generalization.

The main goal of this note is to explicitly describe conditions under which $2$-complexes such as $X$ can be subdivided into nonpositively curved $\vh$-complexes. This develops the ad-hoc method described by Leary. Moreover, it seems that Leary's method might have been guided  by the theory of $\vh$ complexes, and revealing the lurking $\vh$-structure allows us to answer Leary's question. Indeed, since $X$ is a compact nonpositively curved $\vh$-complex, $\pi_1 X = G$ has a so-called quasiconvex hierarchy and since $G$ is word-hyperbolic, it is virtually special and hence residually finite via results from \cite{WiseIsraelHierarchy}.

We now describe the parts of this paper. In Section~\ref{sec:vhcomplexes}, we review material about nonpositively curved $\vh$-complexes. In Section~\ref{sec:squaring}, we describe a criterion for subdividing certain $2$-complex into a $\vh$-complex. Our \emph{squaring construction} is a systematic version of the construction suggested by Leary. Finally, we illustrate this technique in Section~\ref{sec:examples} where we show that Leary's groups are residually finite.

\section{Nonpositively curved $\vh$-complexes}\label{sec:vhcomplexes}
A \emph{square complex} $X$ is a combinatorial $2$-complex whose $2$-cells are squares in the sense that their attaching maps are closed length~$4$ paths in $X^1$. We say $X$ is \emph{nonpositively curved} if $\link(x)$ has girth $\geq 4$ for each $x\in X^0$. Recall that the \emph{link} of a 0-cell $x$ is the graph whose vertices correspond to corners of $1$-cells incident with $x$, and whose edges correspond to corners of $2$-cells incident with $x$. Roughly speaking, $\link(x)$ is isomorphic to the $\epsilon$-sphere about $x$ in $X$.

A \emph{$\vh$-complex} $X$ is a square complex such that the $1$-cells of $X$ are partitioned into two disjoint sets $H$ and $V$ called \emph{horizontal} and \emph{vertical} respectively,
and the attaching map of every $2$-cell is a length~$4$ path that alternates between vertical and horizontal $1$-cells.

A $\vh$-complex is nonpositively curved if and only if there are no length~$2$ cycles in each $\link(x)$ -- indeed,
 each link is bipartite because of the $\vh$~structure. The main result that we will need about $\vh$-complexes is the following result which is a specific case of the main theorem in \cite{WiseIsraelHierarchy}:

\begin{thm}\label{thm:from ultra}
If $X$ is a compact nonpositively curved $\vh$-complex such that $\pi_1X$ is word-hyperbolic then $X$ is virtually special.
Consequently, $\pi_1X$ embeds in $GL_n(\integers)$, and so $\pi_1X$ is residually finite.
\end{thm}

The feature of compact nonpositively curved $\vh$-complexes that enables us to apply Theorem~\ref{thm:from ultra} is that $\pi_1X$ has a so-called quasiconvex hierarchy, which means that it can be built from trivial groups by a finite sequence of HNN extensions and amalgamated free products along quasiconvex subgroups. This type of hierarchy occurs for fundamental groups of $\vh$-complexes because they split geometrically as graphs of graphs as has been explored for instance in \cite{WiseFigure8}.

\section{Squaring polygonal complexes}\label{sec:squaring}
A \emph{bicomplex} is a combinatorial $2$-complex $X$ such that the $1$-cells of $X^1$ are partitioned into two sets called \emph{vertical} and \emph{horizontal}, and the attaching map of each $2$-cell traverses both vertical and horizontal $1$-cells.
We note that each such attaching map is thus a concatenation
$V_1H_1\cdots V_rH_r$ of nontrivial paths where each $V_i$ is the concatenation of vertical $1$-cells and likewise each $H_i$ is a horizontal path.
We refer to the $V_i$ and $H_i$ as \emph{sides} of the 2-cell,
so if the attaching map is
$V_1H_1\cdots V_rH_r$ then
the 2-cell has $2r$ sides.

$X$ has  a \emph{repeated $\vh$-corner} if there is a concatenation $vh$ of a single horizontal and single vertical $1$-cell
that occurs as a ``piece'' in $X$, in the sense that it occurs in two ways as a subpath of attaching maps of $2$-cells. (Specifically, $vh$ or $h^{-1}v^{-1}$ could occur in two distinct attaching maps, or in two ways within the same attaching map.)

A $2$-cell of $X$ with attaching map $V_1H_1\cdots V_rH_r$ satisfies the \emph{triangle inequality} if for each $i$ we have:
$$ |V_i| \leq \sum_{j\neq i} |V_j|   \ \ \text{and} \  \  |H_i| \leq \sum_{j\neq i} |H_j|$$

The goal of this section is to prove the following:
\begin{thm}	[$\vh$ subdivision criterion]\label{thm:bicomplex subdivision}
Let $X$ be a bicomplex.
If each $2$-cell of $X$ satisfies the triangle inequality and if $X$ has no repeated $\vh$-corners,
then $X$ can be subdivided into a nonpositively curved $\vh$-complex.

\noindent\emph{(Hyperbolicity Criterion)} Moreover, if each $2$-cell has at least~$6$ sides, and $X$ is compact, then $\pi_1X$ is hyperbolic.
\end{thm}

\begin{exmp}
It is instructive to indicate some simple examples with no repeated $\vh$-corners but where Theorem~\ref{thm:bicomplex subdivision} fails
without the triangle inequality:
\begin{enumerate}
\item $\langle v,h \mid v^mh^n \rangle$
\item $ \langle v,h \mid v^mh^n, v^{-m}h^n \rangle$
\item  $\langle v,h \mid hv^mh^{-1}v^n \rangle$
\item  $\langle u,v,h \mid h^{-1}uhv^{-1}u^{-1},  h^{-1}vhu^{-1}v^{-1} \rangle$
\end{enumerate}
In each case, $X$ is the standard 2-complex of the presentation given above.
The first example is homeomorphic to a  nonpositively curved $\vh$-complex, but it does not have a $\vh$-complex subdivision that is consistent with the original $\vh$-decomposition of the 1-skeleton.
The second example has no nonpositively curved $\vh$-subdivision -- indeed, the group has torsion when $n,m\neq 0$.
The third example has a nonpositively curved $\vh$-subdivision exactly when $m=\pm n$.
The fourth example is word-hyperbolic. However, it does not have any $\vh$-subdivision. Let us sketch this last claim. Without loss of generality, we may assume that $h$ is horizontal, and then an application of the Combinatorial Gauss-Bonnet Theorem~\cite{GerstenShort91} shows that $u$ and $v$ are vertical. Each two cell in a subdivision has at most four points of positive curvature that must occur at the $\vh$-corners. If there were a $\vh$-subdivision, then each would have curvature exactly $\tfrac\pi 2$ and so the divided two-cell $C'$ would be a rectangular grid. This however is impossible as the length of the path $u$ must be strictly less than the length of the path $v^{-1}u^{-1}$.
\end{exmp}

The main tool used to prove Theorem~\ref{thm:bicomplex subdivision}
is a procedure that subdivides a single polygon $C$ into a $\vh$-complex,
as we can then subdivide $X^1$ and then apply this procedure to subdivide each $2$-cell of $X$.
We will therefore focus on a bicomplex that consists of a single $2$-cell $C$ that we refer to as a polygon.
The observation is that when $C$ satisfies the triangle inequality, there is a pairing  between the vertical edges on $\partial C$, and also a pairing between the horizontal edges on $\partial C$ (the exact type of  ``pairing'' we need is described below).
We add line segments to $C$ that connect the midpoints of paired edges,
and our $\vh$-subdivision of $C$ is simply the dual (see Figure~\ref{fig:TheDual}).

\begin{defn}[Admissible Pairing]
	Let $P$ be a polygon and let $E_1,\dots,E_k$ be distinct sides of $P$. Suppose each edge $E_i$ is subdivided into $|E_i|$ edges. An \emph{admissible pairing} with respect to these sides is an equivalence relation $\sim$ on the subdivided edges satisfying the following:
\begin{enumerate}
	\item If $u\sim v$, $u\in E_i$ and $v\in E_j$ then $i\not= j$.
	\item If $u\sim v$ and $w\sim x$, then $w$ and $x$ lie in the same connected component of $P\setminus u\cup v$.
	\item Each equivalence class has two members.
\end{enumerate}
\end{defn}

We now demonstrate that the triangle inequality implies, and is in fact equivalent to the existence of an admissible pairing.

\begin{lem}\label{main}
	Suppose that $E_1,\dots,E_r$ are sides of a polygon, that each $E_i$ is subdivided into $n_i$ edges,
 and that $\sum n_i$ is even. Then there exists an admissible pairing of these vertices if and only if for each $i$, the inequality $n_i \leq \sum_{j\not= i} n_j$ holds.
\end{lem}

\begin{figure}\centering
\includegraphics[width=190pt]{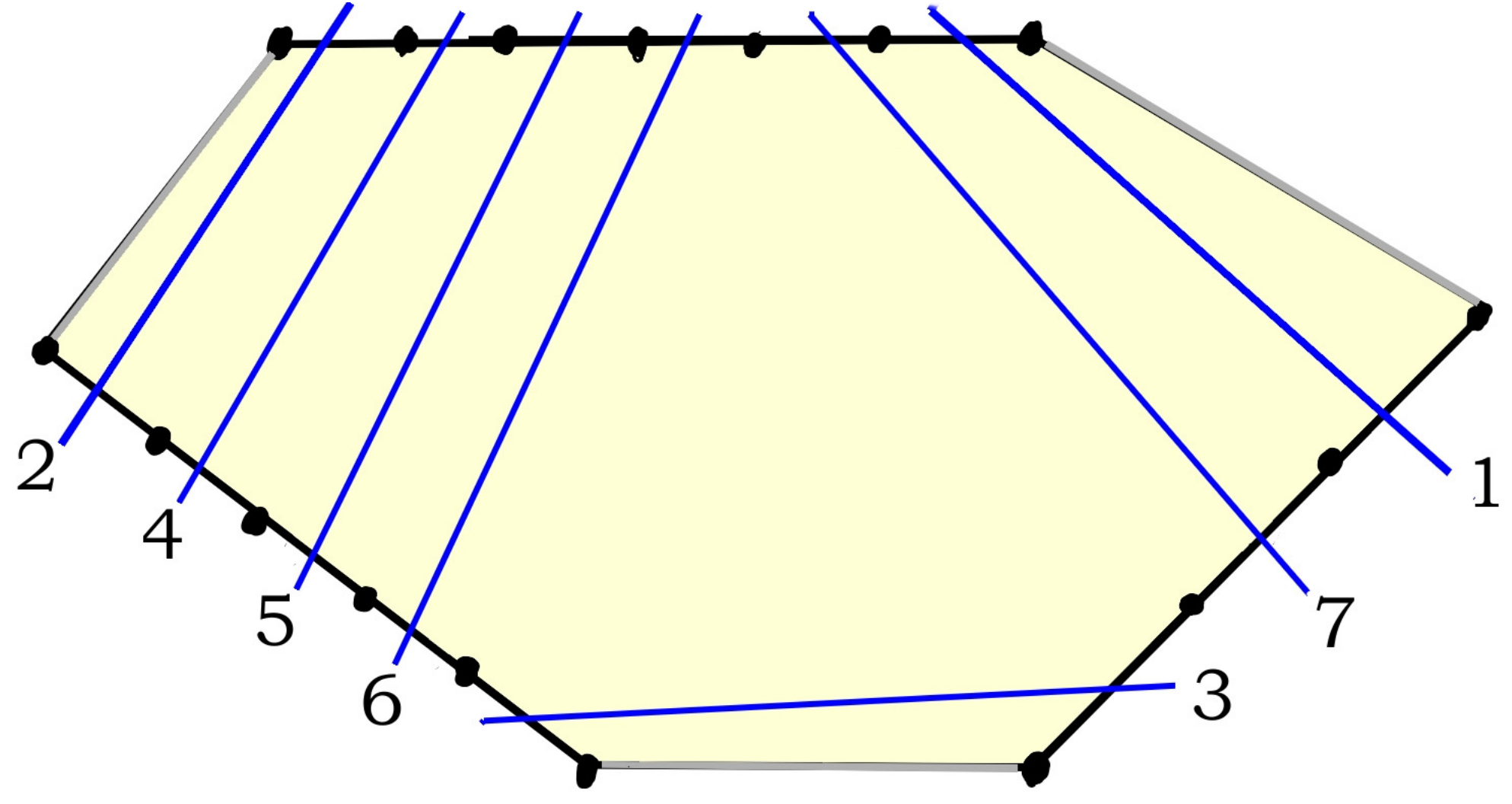}
\caption{\small The sequence of paired edges is indicated by the numbered line segments above.}\label{fig:pairing}
\end{figure}

\begin{proof}
	Suppose that $E_1,\dots, E_r$ are ordered clockwise around the polygon.	We construct the admissible pairing via steps, each step pairing two edges. Each step will consist of choosing $E_i$ with $|E_i|$ maximal, taking an unpaired edge $v\in E_i$ closest to a vertex of the edge, and pairing it with an edge $u$ in some $E_j$ closet to $v$ with $i\not= j$.

By convexity considerations, if we manage a pairing that satisfies (1) and (3) using such a procedure, it will also be admissible. Suppose we represent the number of unpaired edges in each step of this pairing by an ordered vector so that before any edges are paired, our vector is $(n_1,\dots,n_r)$. Using our pairing strategy, each step will consist of subtracting $1$ from two components, one component being maximal.

	It thus suffices to prove that any vector $(n_1,\dots,n_r)$ can be completely reduced to the zero vector if and only if
the triangle inequality $n_i \leq \sum_{j\not= i} n_j$ holds for each $i$.
Suppose that there is an admissible pairing. If the triangle inequality does not hold, then there exists an $i$ such that $n_i > \sum_{j\not=i} n_j$. But each vertex in $V_i$ must be paired with a vertex outside of $V_i$, and this is obviously impossible.

	Now suppose the triangle inequality holds. If $n_i = 1$ for all $i$, then there is an admissible pairing since $\sum n_i$ is even. Otherwise, choose an $i$ such that $n_i \geq n_j$ for each $j$. Subtract $1$ from $n_i$ and an adjacent component. It is easy to verify that the triangle inequality holds in the new vector. Eventually we will get a vector with all entries being $1$ whose sum is even, and so there is an admissible pairing. See Figure~\ref{fig:pairing} for an example of this algorithm applied to a polygon.
\end{proof}

We now prove the main result.
\begin{proof}[Proof of Theorem~\ref{thm:bicomplex subdivision}]

	\noindent\emph{($\vh$ Subdivision Criterion)} By subdividing $X^1$, we can ensure that for each $2$-cell of $X$ with attaching map $V_1H_1\cdots V_rH_r$,
both $\sum |V_i|$ and $\sum |H_i|$ are even.
Note that the triangle inequalities are preserved by subdivision.

By Lemma~\ref{main}, for each $2$-cell $C$, there is an admissible pairing for the vertical $1$-cells of $C$,
and an admissible pairing for the horizontal $1$-cells of $C$.
These pairings form a collection of line segments  within $C$ that join barycenters of paired edges,
and we let $\Gamma$ be the graph consisting of the union of these lines segments.
As in Figure~\ref{fig:TheDual}, the dual of $\Gamma$ within $C$ forms the 1-skeleton of our desired $\vh$-complex $C'$.
We note that $\partial C$ embeds as a subgraph of this dual, and $\partial C \cong \partial C'$.

In this way, we subdivide each $2$-cell of $X$ to obtain a 2-complex $X'$ whose $2$-cells are the subdivisions $C'$ of the various
$2$-cells $C$. Observe that $X'$ is a $\vh$-complex since each $C'$ is a $\vh$-complex and the $\vh$-structure on
$\partial C'$ agrees with the  $\vh$-structure on $X^1$.

Finally, $X'$ is nonpositively curved precisely if each $\link(x)$  has no 2-cycles.
This is clear when $x$ lies in the interior of some $C'$.
When $x\in (X')^1=X^1$, a 2-cycle in $\link(x)$  corresponds precisely to a repeated $\vh$-corner.
\\~\\
\noindent\emph{(Hyperbolicity Criterion)}.  The compact nonpositively curved square complex $X'$ has universal cover $\widetilde X'$ which is $\mathrm{CAT}(0)$, and by Gromov's flat plane theorem \cite{Bridson95}, $\pi_1X'$ is hyperbolic if and only if  $\widetilde X'$ does not contain an isometrically embedded copy of $\Euclidean^2$ -- whose cell structure would be a finite grid in our case.

Observe that each square in a flat plane lies in some $C'$ in the flat plane. The interior $0$-cells of $C'$ must have valence exactly four, and the $0$-cells in the interior of a horizontal or a vertical side must have valence exactly three.

Each $0$-cell on a $\vh$-corner must have even valence either $2$ or $4$, but again the latter is impossible for the embedding in the flat plane shows that there is a repeated $\vh$-corner. Furthermore, since there are at least six sides, there are at least six $\vh$-corners. Combinatorial Gauss-Bonnet applied to $C'$ shows that
\begin{align*}
	2\pi &= \sum_{v\in\mathrm{Int}(C)}\kappa(v) + \sum_{v\in\partial(C))} \kappa(v)\geq 6\cdot\frac \pi 2  = 3\pi.
\end{align*}
  Thus $C'$ cannot lie in an infinite grid.
     \end{proof}

\begin{figure}\centering
\includegraphics[width=300pt]{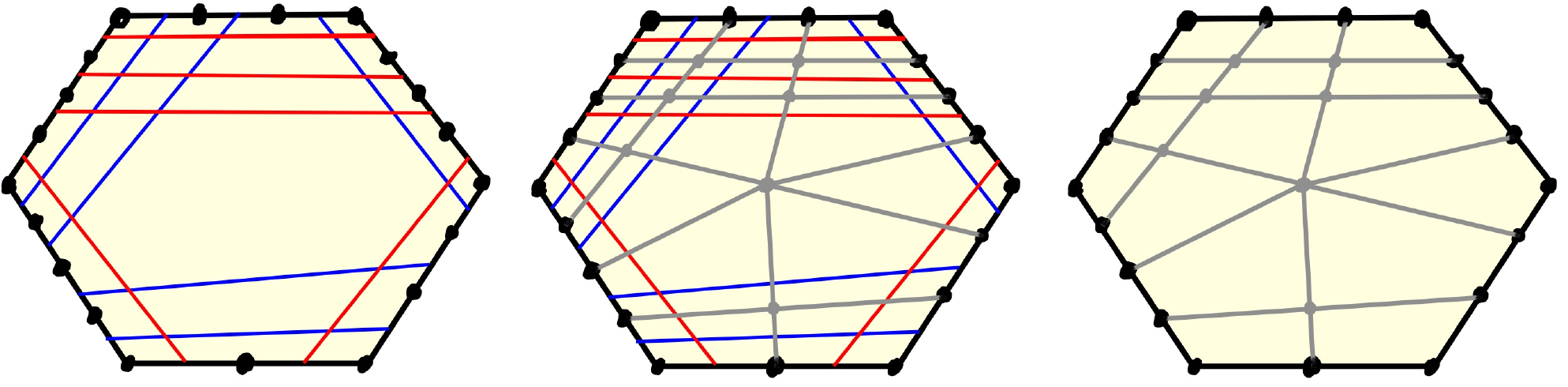}
\caption{\small The two systems of dual curves in $C$ are indicated on the left, the dual of $\Gamma$ is indicated in the middle,
  and the $\vh$-subdivision $C'$ is indicated on the right.}\label{fig:TheDual}
\end{figure}

\section{Application to Leary's examples}\label{sec:examples}
There are two examples from Leary's paper \cite{Leary_KanThurston} that we shall consider. The first example is the group:
\begin{equation}
	\left\langle a,b,c,d,e,f  \bigg|
  \begin{array}{ccc}
     abcdef, & ab\iv c^2 f\iv e^2 d\iv, & a^2fc^2bed, \\
    ad^{-2}cb^{-2}ef\iv, & af^{-2}cd\iv eb^{-2}, & ad^2cf^2eb^2
  \end{array}
\right\rangle
\end{equation}
Leary proved that this group is nontrivial, torsion-free, and acyclic, and asked whether it is also residually finite.

The standard 2-complex $X$ of this presentation is a bicomplex whose vertical $1$-cells correspond to $\{a,c,e\}$ and whose horizontal $1$-cells correspond to $\{b,d,f\}$. It is easily verified that the $2$-cells satisfy the triangle inequality.

Finally, this group is hyperbolic because each polygon has at least six sides, so we can apply the hyperbolicity criterion in Theorem~\ref{thm:bicomplex subdivision}. The group is therefore residually finite by  Theorem~\ref{thm:from ultra}.

A second family of examples with which Leary is concerned  is defined as follows. We let $n\in\mathbb{N}$ with $n\geq 4$, and for each $i\in \integers/n\integers$ we define the two words $A_i = a_ia_{i+2}a_i^{-2}a_{i+2}^{-1}a_i$ and $B_i = b_ib_{i+2}b_i^{-2}b_{i+2}^{-1}b_i$ and the eight words given by:
$$	a_iA_iB_iA_{i+1}B_iA_{i+2}B_iA_{i+3}B_i  \ \ \text{ and }  \ \ 	b_iB_iA_i^{-1}B_iA_{i+1}^{-1}B_iA_{i+2}^{-1}B_iA_{i+3}^{-1}
$$
This group is again hyperbolic by our hyperbolicity criterion (this also follows since the presentation is $C'(\frac 1 6)$). The triangle inequalities are readily verified, and there are no repeated $\vh$-corners by design.

{\bf Acknowledgement:} We are grateful to the referee for improving the exposition of this paper.

\bibliographystyle{alpha}
\bibliography{wise}
\end{document}